\documentclass[12pt]{amsart}
\usepackage{amssymb, amsfonts}
\usepackage{hyperref}

\usepackage[retainorgcmds]{IEEEtrantools}

\renewcommand{\a}{\alpha}
\renewcommand{\b}{\beta}
\renewcommand{\d}{\delta}

\newcommand{\g}{\gamma}
\newcommand{\G}{\Gamma}
\renewcommand{\l}{\lambda}
\renewcommand{\L}{\Lambda}
\newcommand{\s}{\sigma}

\renewcommand{\o}{\omega}

\renewcommand{\t}{\tau}

\newcommand{\cA}{{\mathcal A}}
\newcommand{\cB}{{\mathcal B}}

\newcommand{\cD}{{\mathcal D}}

\newcommand{\cH}{{\mathcal H}}

\newcommand{\cL}{{\mathcal L}}

\newcommand{\bC}{{\mathbb C}}

\newcommand{\rtr}{{\mathrm{tr}}}

\newcommand{\Lat}{{Latr\'emolie\`re}}

\newtheorem{thm}{Theorem}[section]
\newtheorem{cor}[thm]{Corollary} 
\newtheorem{prop}[thm]{Proposition} 

\theoremstyle{definition}   %my change

\newtheorem{defn}[thm]{Definition} 
 
\newtheorem{notation}[thm]{Notation} 

\numberwithin{equation}{section}

\title
[Matricial bridges]
{Matricial bridges for ``Matrix algebras converge to the sphere''  \\
}
\author{Marc A. Rieffel}
\address{Department of Mathematics \\
University of California \\ Berkeley, CA 94720-3840}
\email{rieffel@math.berkeley.edu}
\thanks{The research reported here was
supported in part by National Science Foundation grant DMS-1066368.
}

\begin{document}

\begin{abstract}
In the high-energy quantum-physics literature one finds 
statements such as ``matrix algebras converge to the sphere''.
Earlier I provided a general setting for understanding such
statements, in which the matrix algebras are viewed as
quantum metric spaces, and convergence is with respect
to a quantum Gromov-Hausdorff-type distance. In the present
paper, as preparation of discussing similar statements for
convergence of ``vector bundles'' over matrix algebras to
vector bundles over spaces, we introduce and study
suitable matrix-norms for matrix algebras and spaces.
Very recently \Lat \ introduced an improved
quantum Gromov-Hausdorff-type distance between quantum 
metric spaces. We use it throughout this paper.
To facilitate the calculations we introduce and develop 
a general notion of
``bridges with conditional expectations''.
\end{abstract}

\maketitle
\allowdisplaybreaks

%\large{

\section{Introduction}
In several earlier papers \cite{R7, R25, R21} I showed how to give 
a precise meaning to statements in the literature of high-energy
physics and string theory of the kind ``matrix algebras 
converge to the sphere''. (See the references to the quantum physics
literature given in 
\cite{R7, R17, Stn2, ChZ, DiG, AHI, AcD}.) I did this by introducing
and developing a concept of ``compact quantum metric spaces'', and a corresponding quantum Gromov-Hausdorff-type distance 
between them. The compact quantum spaces are unital C*-algebras,
and the metric data is given by putting on the
algebras seminorms that play the role of the usual Lipschitz
seminorms on the algebras of continuous functions on ordinary 
compact metric spaces. The natural setting for ``matrix algebras 
converge to the sphere'' is that of coadjoint orbits of compact
semi-simple Lie groups.

But physicists need much more than just the algebras. They need vector
bundles, gauge fields, Dirac operators, etc. So I now seek to give
precise meaning to statements in the physics literature of the kind
``here are the vector bundles over the matrix algebras that
correspond to the monopole bundles on the sphere''. (See \cite{R17}
for many references.) In \cite{R17} I studied convergence of ordinary 
vector bundles on ordinary compact metric spaces for ordinary 
Gromov-Hausdorff distance. From that study it became clear that one
needed Lipschitz-type seminorms on all the matrix algebras over the
underlying algebras, with these seminorms coherent in the
sense that they form a ``matrix seminorm'' (defined below).
The purpose of this paper is to define and develop the properties 
of such matrix seminorms for the setting of coadjoint orbits, and
especially to study how these matrix seminorms mesh with
quantum Gromov-Hausdorff distance.

Very recently \Lat \ introduced an improved version of quantum
Gromov-Hausdorff distance \cite{Ltr2} that he calls ``propinquity''. We
show that propinquity works very well for our setting of coadjoint orbits,
and so propinquity is the form of quantum Gromov-Hausdorff
distance that we use in this paper. \Lat \ defines his propinquity
in terms of an improved version of the ``bridges'' that I had used
in my earlier papers. For our matrix seminorms we need
corresponding ``matricial bridges'', and we show how to
construct natural ones for the setting of coadjoint orbits.

It is crucial to obtain good upper bounds for the lengths of
the bridges that we construct. In the matricial setting the
calculations become somewhat complicated. In order to ease
the calculations we introduce a notion of ``bridges with 
conditional expectations'', and develop their general
theory, including the matricial case, and including
bounds for their lengths in the matricial case. 

The main theorem of this paper, Theorem \ref{thmmat1q}, states
in a quantitative way that for the case of coadjoint orbits
the lengths of the matricial bridges goes to 0 as the size
of the matrix algebras goes to infinity.

We also discuss a closely related class of examples
coming from \cite{R25}, for which we construct bridges 
between different matrix algebras associated to a given 
coadjoint orbit. This provides further motivation for our 
definitions and theory of bridges with conditional
expectation.

\tableofcontents
 
%%%%%%%%%%%%%%%%%%%%%%%%%%%%%%%%%%%%%%%%%%%

\section{The first basic class of examples}
\label{basic1}

In this section we describe the first of the two basic classes of examples 
underlying this paper.
It consists of the main class of examples studied in the 
papers \cite{R7, R21}. 
We begin by describing the common setting for the two 
basic classes of examples.

Let $G$ be a compact group (perhaps even finite, at first).  Let $U$ be an
irreducible unitary representation of $G$ on a 
(finite-dimensional) Hilbert space ${\mathcal
H}$.  Let $\cB = \cL({\mathcal H})$ denote the $C^*$-algebra of all linear
operators on ${\cH}$ (a ``full matrix algebra'', with its operator
norm).  There is a natural action, $\a$, of $G$ on $\cB$ by conjugation by
$U$, that is, $\a_x(T) = U_xTU_x^*$ for $x \in G$ and $T \in \cB$.  Because
$U$ is irreducible, the action $\a$ is ``ergodic'', in the sense that the
only $\a$-invariant elements of $\cB$ are the scalar multiples of the
identity operator. 

Let $P$ be a rank-one projection in $\cB({\cH})$ (traditionally
specified by giving a non-zero vector in its range).  For any $T \in \cB$ we
define its Berezin covariant symbol \cite{R7},  $\s_T$, with
respect to $P$, by
\[
\s_T(x) = \rtr(T\a_x(P)),
\]
where $\rtr$ denotes the usual (un-normalized) trace on $\cB$.  (When the
$\a_x(P)$'s are viewed as giving states on $\cB$ via $\rtr$, they
form a family of ``coherent states'' \cite{R7} if a few additional
conditions are satisfied.)  Let $H$ denote the
stability subgroup of $P$ for $\a$.  Then it is evident that $\s_T$ can be
viewed as a (continuous) function on $G/H$.  We let $\l$ denote the action
of $G$ on $G/H$, and so on $\cA = C(G/H)$, by left-translation.  If we 
note that
$\rtr$ is $\a$-invariant, then it is easily seen that $\s$ is a unital,
positive, norm-nonincreasing, $\a$-$\l$-equivariant map from $\cB$ into $\cA$.
%In the next section????? 
%we will show how to construct a bridge, in the sense of
%\Lat, from $\cA$ to $\cB$.

 Fix a continuous length function, $\ell$, on $G$ (so
$G$ must be metrizable).  Thus $\ell$ is non-negative, $\ell(x) = 0$ iff
$x = e_G$ (the identity element of $G$), $\ell(x^{-1}) = \ell(x)$, and
$\ell(xy) \le \ell(x) + \ell(y)$. We also require 
that $\ell(xyx^{-1}) = \ell(y)$ for all $x$, $y \in G$.  Then in
terms of $\a$ and $\ell$ we can define a 
seminorm, $L^\cB$, on $\cB$ by the formula
\begin{equation}
\label{lipn}
L^\cB(T) = \sup\{ \|\a_x(T) - T \|/\ell(x): x \in G \quad \mathrm{and} 
\quad x \neq e_G\}  .
\end{equation}
Then $(\cB,L_\cB)$ 
is an example of a compact C*-metric-space, 
as defined in definition 4.1 of \cite{R21}, 
and in particular $L_\cB$ satisfied the conditions
given there for being a ``Lip-norm''. 

Of course, from $\l$ and $\ell$ we also obtain a seminorm, 
$L^\cA$, on $\cA$ by the evident analog of
formula \ref{lipn}, except that we must permit $L^\cA$
to take the value $\infty$. It is shown in proposition 2.2
of \cite{R4} that the set of functions for which
$L^\cA$ is finite (the Lipschitz functions) is a dense
$*$-subalgebra of $\cA$.  Also, $L^\cA$ 
is the restriction to $\cA$ of the seminorm on
$C(G)$ that we get from $\ell$ when we view $C(G/H)$ as a subalgebra of
$C(G)$, as we will often do when convenient.  
%We will often not restrict
%$L^\cA$ to the Lipschitz functions, but rather permit $L_\cA(f) = +\infty$.  
From $L^\cA$ we can use equation \ref{metr} below
to recover the usual quotient metric \cite{Wvr2}
on $G/H$ coming from the metric on $G$ determined by $\ell$.  
One can check easily that $L^\cA$ in turn comes from
this quotient metric.  Thus $(\cA,L^\cA)$ is the compact C*-metric-space
associated to this ordinary compact metric space. 
Then for any bridge
from $\cA$ to $\cB$ we can use $L^\cA$ and $L^\cB$ to define the
length of the bridge in the way given by \Lat, which we will describe
soon below. 

For any two unital C*-algebras $\cA$ and $\cB$ a
bridge from  $\cA$ to $\cB$ in the sense of \Lat \ \cite{Ltr2} is a quadruple
$\{\cD, \pi_\cA, \pi_\cB, \o \}$ for which $\cD$ is a unital C*-algebra,
$\pi_\cA$ and $\pi_\cB$ are unital injective homomorphisms of
$\cA$ and $\cB$ into $\cD$, and $\o$ is a self-adjoint element
of $\cD$ such that 1 is an element of the spectrum
of $\o$ and  $\|\o \| = 1$. Actually, \Lat \ only requires a looser 
but more complicated condition on $\o$, but the
above condition will be appropriate for our examples. Following
\Lat \ we will call $\o$ the ``pivot'' for the bridge.
We will often omit mentioning the injections $\pi_\cA$ and $\pi_\cB$
when it is clear what they are from the context, and accordingly we
will often write as though $\cA$ and $\cB$ are unital subalgebras of $\cD$.

For our first class of examples, in which $\cA$ and $\cB$ are 
as described in the paragraphs above, 
we take $\cD$ to be the C*-algebra 
\[
\cD = \cA \otimes \cB = C(G/H, \cB)  .
\]
We take $\pi_\cA$ to be the injection of $\cA$ into $\cD$
defined by
\[
\pi_\cA(a) = a \otimes 1_\cB
\]
for all $a \in \cA$, where $1_\cB$ is the identity element of $\cB$. 
The injection $\pi_\cB$ is defined similarly.
From the many calculations done in \cite{R7, R21} it is
not surprising that we define the pivot $\o$ to be the function in  
$C(G/H, \cB)$ defined by
\[
\o(x) = \a_x(P)
\]
for all $x \in G/H$. We notice that $\o$ is actually a projection
in $\cD$, and so it satisfies the requirements for being a pivot.
We will denote the bridge $\{\cD, \o\}$ by $\Pi$.

For any bridge between two unital C*-algebras $\cA$ and $\cB$
and any choice of
seminorms $L^\cA$ and $L^\cB$ on $\cA$ and $\cB$,
\Lat \  \cite{Ltr2} defines the ``length'' of the
bridge in terms of these seminorms. For this
he initially puts relatively weak requirements on the seminorms,
but for the purposes of the matricial bridges that we will
define later, we need somewhat different weak requirements.
To begin with, \Lat \ only requires his seminorms, say
$L^\cA$ on a unital C*-algebra $\cA$, to be
defined on the subspace of self-adjoint elements of the
algebra, but we need $\cA$ to be defined on
all of $\cA$. To somewhat compensate for this we require
that $L^\cA$ be a  $*$-seminorm. As with \Lat,
our $L^\cA$ is permitted to take value $+\infty$.
\Lat \  also requires the subspace on which $L^\cA$
takes finite values to be dense in the algebra. We do not really
need this here, but for us there would be no harm in assuming it,
and all interesting examples probably will satisfy this. Finally,
\Lat \ requires that the null space of $L^\cA$ (i.e
where it takes value 0) be exactly $\bC 1_\cA$.
We must loosen this to simply requiring that $L^\cA(1_\cA) = 0$,
but permitting $L^\cA$ to also take value 0 on elements not
in $\bC 1_\cA$. We think of such seminorms as ``semi-Lipschitz
seminorms''. To summarize all of this we make:

\begin{defn}
\label{slip}
By a \emph{slip-norm} on a unital C*-algebra $\cA$ we mean
a $*$-seminorm, $L$, on $\cA$ that is permitted to take the
value $+\infty$, and is such that $L(1_\cA) = 0$.
\end{defn}

Because of these weak requirements on $L^\cA$, various
quantities in this paper may be $+\infty$, but most interesting
examples will satisfy stronger requirements that will 
result in various quantities being finite.

\Lat \ defines the length of a bridge by first defining its ``reach''
and its ``height''. We apply his definitions to slip-norms.

\begin{defn}
\label{reach}
Let $\cA$ and $\cB$ be unital C*-algebras and let
$\Pi = \{\cD, \o\}$ be a bridge from $\cA$ to $\cB$ . 
Let $L^\cA$ and $L^\cB$ be slip-norms on $\cA$ and $\cB$. Set 
\[
\cL_\cA^1 = \{a \in \cA:  a = a^* \quad \mathrm{and} \quad L^\cA(a) \leq 1\}   ,
\]
and similarly for $\cL_\cB^1$. (We can view these as subsets of $\cD$.) 
Then the \emph{reach} of $\Pi$ is given by:
\[
\mathrm{reach}(\Pi) = \mathrm{Haus}_\cD\{\cL_\cA^1  \o \ , \     \o  \cL_\cB^1\}   ,
\]
where $\mathrm{Haus}_\cD$ denotes the Hausdorff distance with respect
to the norm of $\cD$, and where the product defining $\cL_\cA^1  \o$ and
$ \o  \cL_\cB^1$ is that of $\cD$.  
\end{defn}

\Lat \ shows
just before definition 3.14 of \cite{Ltr2} that, 
under conditions that include the case in which 
$(\cA, L^\cA)$ and $(\cB, L^\cB)$ are 
C*-metric spaces, the reach of $\Pi$ is finite. 

To define the height of $\Pi$ we need to consider the state space, $S(\cA)$,
of $\cA$, and similarly for $\cB$ and $\cD$. Even more, we set
\[
S_1(\o) = \{\phi \in S(\cD): \phi(\o) = 1\}     ,   
\]
the ``level-1 set of $\o$''. The elements of $S_1(\o)$ are ``definite'' 
on $\o$ in the sense \cite{KR1} that 
for any $d \in \cD$ we have
\[
\phi(d\o) = \phi(d) = \phi(\o d)     .
\]
Let $\rho_\cA$ denote the metric on $S(\cA)$ determined by $L^\cA$
by the formula 
\begin{equation}
\label{metr}
\rho_\cA(\mu, \nu) = \sup\{|\mu(a) - \nu(a)|: L^\cA(a) \leq 1\}.
\end{equation}
(Without further conditions on $L_\cA$ we must
permit $\rho_\cA$ to take the value $+\infty$. Also, it is
not hard to see that the supremum can be taken equally well
just over $\cL^1_\cA$.)
Define $\rho_\cB$ on $S(\cB)$ similarly. 
\begin{notation}
We denote by $S_1^\cA(\o)$ the restriction of the
elements of $S_1(\o)$ to $\cA$. We define $S_1^\cB(\o)$ 
similarly. 
\end{notation}

\begin{defn}
\label{height}
Let $\cA$ and $\cB$ be unital C*-algebras and let
$\Pi = \{\cD, \o\}$ be a bridge from $\cA$ to $\cB$ . 
Let $L^\cA$ and $L^\cB$ be slip-norms on $\cA$ and $\cB$. 
The \emph{height} of the bridge $\Pi$ is given by
\[
\mathrm{height}(\Pi) =
\max\{\mathrm{Haus}_{\rho_\cA}(S_1^\cA(\o), S(\cA)) , \ 
\mathrm{Haus}_{\rho_\cB}(S_1^\cB(\o) , S(\cB))\}  ,
\]
where the Hausdorff distances are with respect to the indicated
metrics (and value $+\infty$ is allowed). 
The length of $\Pi$ is then defined by
\[
\mathrm{length}(\Pi) = \max\{\mathrm{reach}(\Pi), \mathrm{height}(\Pi)\}  .
\]
\end{defn}

In Section \ref{app1} we will show how to obtain a useful upper bound
on the length of $\Pi$ for our first class of examples.

%%%%%%%%%%%%%%%%%%%%%%%%%%%%%%%%

\section{The second basic class of examples}
\label{basic2}

Our second basic class of examples has the same starting point
as the first class, consisting of $G$, $\cH$, $U$ and $P$ as before,
with $\cB = \cL(\cH)$. But now we will also have a second
irreducible representation. The more concrete class of examples
motivating this situation, but for which we will not need the details,
is that in \cite{R7, R25, R21} in 
which $G$ is a compact semi-simple Lie group, $\l$ is
a positive integral weight, and our two representations of $G$
are the representations with highest weights $m\l$ and $n\l$
for positive integers $m$ and $n$, $m \neq n$. 
Furthermore, the projections
$P$ are required to be those along highest weight vectors.  The
key feature of this situation that we do need to remember here
is that the stability subgroups $H$ for the two projections
coincide.

Accordingly, for our slightly more general situation, we will denote
our two representations by $(\cH^m, U^m)$ and $(\cH^n, U^n)$,
where now $m$ and $n$ are just labels. Our two C*-algebras will
be $\cB^m = \cL(\cH^m)$ and $\cB^n = \cL(\cH^n)$. We will
denote the action of $G$ on these two algebras just by $\a$,
since the context should always make clear which algebra
is being acted on. The
corresponding projections will be $P^m$ and $P^n$. The crucial
assumption that we make is that the stability subgroups of
these two projections coincide. We will denote this common
stability subgroup by $H$ as before.

We construct a bridge from $\cB^m$ to $\cB^n$ as follows.
We let $\cA = C(G/H)$ as in our first class of examples, 
and we define $\cD$ by
\[
\cD = \cB^m \otimes \cA \otimes \cB^n \ = \ C(G/H, \cB^m \otimes \cB^n)  .
\]
We view $\cB^m$ as a subalgebra of $\cD$ by sending
$b \in \cB^m$ to $b \otimes 1_\cA \otimes 1_{\cB^n}$, and similarly for $\cB^n$. 
From the many calculations done in \cite{R25} it is
not surprising that we define the pivot, $\o$, to be the function in  
$C(G/H, \cB^m \otimes \cB^n)$ defined
by
\[
\o(x) = \a_x(P^m) \otimes \a_x(P^n)  .
\] 

We let $L^m$ be the Lip-norm defined on $\cB^m$ determined 
by the action $\a$ and the length function $\ell$ as in 
Section \ref{basic1}, and similarly for $L^n$ on $\cB^n$. In
terms of these Lip-norms the length of any bridge from
$\cB^m$ to $\cB^n$ is defined. Thus the length of the bridge
described above is defined. In Section \ref{app2} we will see how 
to obtain useful upper bounds on the length of this bridge.

%%%%%%%%%%%%%%%%%%%%%%%%%%%%%%%%%%%%%%%%%%%%%%%%%%%%%%

\section{Bridges with conditional expectations}
\label{brce}
We will now seek a somewhat general framework for obtaining 
useful estimates for the lengths of bridges such as
those of our two basic classes of examples.
To discover this framework we will explore some properties
of our two basic classes of examples.
We will summarize what we find at the end of this section.

On $G/H$ there is a unique probability measure that is invariant
under left translation by elements of $G$. We denote the
corresponding linear functional on $\cA = C(G/H)$ by $\t_\cA$,
and sometimes refer to it as the canonical tracial state on $\cA$.
On $\cB = \cL(\cH)$ there is a unique tracial state, which we
denote by $\t_\cB$. These combine to form a tracial state,
$\t_\cD = \t_\cA \otimes \t_\cB$ on $\cD = \cA \otimes \cB$. Similarly, 
we have the unique tracial
states $\t_m$ and $\t_n$ on $\cB^m$ and $\cB^n$, which
combine with $\t_\cA$ to give a tracial state on
$\cD = \cB^m \otimes \cA \otimes \cB^n$. 

For $\cD = \cA \otimes \cB$, the tracial state $\t_\cB$
determines a conditional expectation, $E^\cA$,
from $\cD$ onto its subalgebra $\cA$, defined on
elementary tensors by
\[
E^\cA(a \otimes b) = a \t_\cB(b)
\] 
for any $a \in \cA$ and $b \in \cB$. (This is an example of
a ``slice map'' as discussed in \cite{Blk2}, where
conditional expectations
are also discussed.) This conditional expectation
has the property that for any $d \in \cD$ we have
\[
\t_\cA(E^\cA(d)) = \t_\cD(d)  ,
\]
and it is the unique conditional expectation with this property. (See
corollary II.6.10.8 of \cite{Blk2}.) 
In the same way the tracial state $\t_\cA$
determines a canonical conditional expectation, $E^\cB$ from
$\cD$ onto its subalgebra $\cB$.

For the case in 
which $\cD = \cB^m \otimes \cA \otimes \cB^n$, 
the tracial state $\t_\cA \otimes \t_n$ on $\cA \otimes \cB^n$
determines a canonical conditional
expectation, $E^m$, from $\cD$ onto $\cB^m$ in the same 
way as above,
and the tracial state $\t_m \otimes \t_\cA$ determines
a canonical conditional expectation, $E^n$, from $\cD$ onto
$\cB^n$ .

These conditional expectations relate well to the pivots of the
bridges. For the case in which $\cD = \cA \otimes \cB$ we find
that for any $F \in \cD = C(G/H, \cB)$ we have
\[
E^\cA(F\o)(x) = \t_\cB(F(x)\a_x(P))  .
\]
In particular, for any $T \in \cB$ we have
\[
E^\cA(T\o)(x) = \t_\cB(T\a_x(P))
\]
for all $x \in G/H$. Aside from the fact that we are here using the
normalized trace instead of the standard trace on the matrix
algebra $\cB$, the right-hand side is exactly the definition of
the Berezin covariant symbol of $T$ that plays such an important
role in \cite{R7, R21} (beginning in section 1 of \cite{R7}), 
and that is denoted there by
$\s_T$. This indicates that for general $\cA$ and $\cB$ a
map $b \mapsto E^\cA(b\o)$ might be of importance to us.
For our specific first basic class of examples we note 
the following favorable properties:
\begin{enumerate}
\item Self-adjointness \quad i.e. $E^\cA(F^*\o) = (E^\cA(\o F))^*$
 for all $F \in \cD$.
\item $E^\cA(F\o) = E^\cA(\o F)$ \quad for all $F \in \cD$.
\item Positivity, \quad i.e. if $F \geq 0$ then $E^\cA(F\o) \geq 0$.
\item $E^\cA(1_\cD\o) = r^{-1}1_\cA$ \quad where $\cB$ is an $r \times r$ 
matrix algebra   .
\end{enumerate}

However, if we consider $E^\cB$ instead $E^\cA$, 
then for any $F \in \cD$ we have
\[
E^\cB(F\o) = \int_{G/H} F(x)\a_x(P) \ dx  ,
\]
and we see that in general properties 1-3 above fail, although property
4 still holds, with the same constant $r$. But if we restrict $F$ to be
any $f \in \cA$, we see that properties 1-3 again hold. Even more,
the expression 
\[
\int_{G/H} f(x)\a_x(P) \ dx
\]
is, except for normalization of the trace, the formula involved in the Berezin
contravariant symbol that in \cite{R7} is denoted by $\breve \s$.

For our second class of examples, in which 
$\cD = \cB^m \otimes \cA \otimes \cB^n$, we find that for
$F \in \cD = C(G/H, \cB^m \otimes \cB^n)$ we have
\[
E^m(F\o) = \int_{G/H} (\iota_\cA 
\otimes \t_n)(F(x)(\a_x(P^m) \otimes \a_x(P^n))) \ dx  .
\]
Again we see that properties 1-3 above are not in general satisfied.
But if we restrict $F$ to be any $T \in \cB^n$ then the above 
formula becomes
\[
\int_{G/H} \a_x(P^m) \t_n(T\a_x(P^n)) \ dx   ,
\]
which up to normalization of the trace
is exactly the second displayed formula in section 3 of \cite{R25}.
It is not difficult to
see that properties 1-3 above are again satisfied under this restriction.

We remark that it is easily seen that the maps 
$T \mapsto E^m(T\o)$ from $\cB^n$ to $\cB^m$ and
$S \mapsto E^n(S\o)$ from $\cB^m$ to $\cB^n$ are
each other's adjoints when they
are viewed as being between the Hilbert spaces $\cL^2(\cB^m, \t_m)$
and $\cL^2(\cB^n, \t_n)$. A similar statement hold for our first basic
class of examples.

With these observations in mind, we begin to formulate a somewhat
general framework. As before, we assume that we 
have two unital C*-algebras
$\cA$ and $\cB$, and a bridge $\Pi = (\cD, \o)$ from $\cA$ to
$\cB$. We now require that we are given conditional expectations
$E^\cA$ and $E^\cB$ from $\cD$ onto its subalgebras $\cA$ and
$\cB$. (We do not require that they be associated to any tracial
states.) We require that they relate well to $\o$. To begin with,
we will just require that $\o \geq 0$ so that $\o^{1/2}$ exists.
Then the map
\[
D \mapsto E^\cA(\o^{1/2}D\o^{1/2})
\]
from $\cD$ to $\cA$ is positive.

Once we have slip-norms $L^\cA$ and $L^\cB$ on $\cA$ and $\cB$,
we need to require that the conditional expectations are 
compatible with these slip-norms. To begin with, we
require that if $L^\cB(b) = 0$ for some $b \in \cB$
then $L^\cA(E^\cA(\o^{1/2}b\o^{1/2})) = 0$.
But one of the conditions on a Lip-norm is that it takes
value 0 exactly on the scalar multiples of the identity element,
and the case of Lip-norms is important to us. For Lip-norms
we see that the above requirement implies that
$E^\cA(\o) \in \bC 1_\cA$, and so $E^\cA(\o) = r_\o1_\cA$
for some positive real number $r_\o$. We require the same
of $E^\cB$ with the same real number, so that we require that
\[
E^\cA(\o) = r_\o1_\cA = E^\cB(\o)  .
\]
We then define a map, $\Phi^\cA$, from $\cD$ to $\cA$ by
\[
\Phi^\cA(d) = r_\o^{-1} E^\cA(\o^{1/2}d\o^{1/2})   .
\]
In a similar way we define $\Phi^\cB$ from $\cD$ to $\cB$.
We see that $\Phi^\cA$ and $\Phi^\cB$ are unital positive maps, and
so are of norm 1 (as seen by composing them with states). 
Then the main compatibility requirement that we need
is that for all $b \in \cB$ we have
\[
L^\cA(\Phi^\cA(b)) \leq L^\cB(b)   ,
\]
and similarly for $\cA$ and $\cB$ reversed. Notice that this
implies that if $b \in \cL^1_\cB$ then $\Phi^\cA(b) \in \cL^1_\cA$.

We now show how to obtain an upper bound for the reach of
the bridge $\Pi$ when the above requirements are satisfied.
Let $b \in \cL^1_\cB$ be given. As an approximation to $\o b$ by
an element of the form $a\o$ for some $a \in \cL^1_\cA$ we take
$a = \Phi^\cA(b)$. It is indeed in $\cL^1_\cA$ by the 
requirements made just above. This prompts us to set
\begin{equation}
\label{gamb1}
\g^\cB = 
\sup\{\|\Phi^\cA(b)\o - \o b\|_\cD :  b \in \cL^1_\cB \}  ,
\end{equation}
and we see that $\o b$ is then in the $\g^\cB$-neighborhood
of $\cL^1_\cA \o$. Note that without further assumptions
on $L^\cB$ we could have $\g^\cB = +\infty$.
Interchanging the roles of $\cA$ and $\cB$, we define $\g^\cA$
similarly. We then see that
\[
\mathrm{reach}(\Pi) \leq \max\{\g^\cA, \g^\cB\}  .
\]
We will explain in Sections \ref{app1} and \ref{app2} 
why this upper bound is
useful in the context of \cite{R7, R25, R21}.

We now consider the height of $\Pi$. For this we need to consider $S_1(\o)$
as defined in Section \ref{basic1}. Let $\mu \in S(\cA)$. Because 
$\Phi^\cA$ is positive and unital, its
composition with $\mu$ is in $S(\cD)$. When we
evaluate this composition at $\o$ to see if it is in $S_1(\o)$, we
obtain $\mu(r_\o^{-1} E^\cA(\o^2))$, and we need this to equal 1.
Because $\mu(r_\o^{-1} E^\cA(\o)) = 1$, it follows
that we need $\mu(r_\o^{-1} E^\cA(\o-\o^2)) = 0$. If this is to hold
for all $\mu \in S(\cA)$, we must have $E^\cA(\o-\o^2) = 0$.
If $E^\cA$ is a faithful conditional expectation, as is true for
our basic examples, then because $\o \geq \o^2$ 
it follows that $\o^2 = \o$ so that
$\o$ is a projection, as is also true for our basic examples.
These arguments are reversible, and so it is easy to see
that if $\o$ is a projection, then for every $\mu \in S(\cA)$
we obtain an element, $\phi_\mu$, of $S_1(\o)$, defined by
\[
\phi_\mu(d) = \mu(r_\o^{-1} E^\cA(\o d\o)) = \mu(\Phi^\cA(d))  .
\]  
This provides us with a substantial collection of elements
of $S_1(\o)$.

Consequently, since to estimate the height of $\Pi$ we need to
estimate the distance from each $\mu \in S(\cA)$ to
$S_1^\cA(\o)$, we can hope that $\phi_\mu$ restricted
to $\cA$ is relatively close to $\mu$. Accordingly, for any
$a \in \cA$ we
compute
\[
|\mu(a) - \phi_\mu(a)| = |\mu(a) - \mu(\Phi^\cA(a))|
\leq \|a - \Phi^\cA(a)\|  . 
\]
Set 
\[
\d^\cA = \sup \{\|a - \Phi^\cA(a)\| : a \in \cL^1_\cA\}.
\]
Then we see that
\[
\rho_{L^\cA}(\mu, \ \phi_\mu |_\cA) \leq \d^\cA  .
\]
We define $\d^\cB$ in the same way, and obtain the
corresponding estimate for the distances from elements
of $S(\cB)$ to the restriction of $S_1(\o)$ to $\cB$. 
In this way we see that
\[
\mathrm{height}(\Pi) \leq \max\{\d^\cA, \d^\cB\}  .
\]
(Notice that $\d^\cA$ involves what $\Phi^\cA$ does on $\cA$,
whereas $\g^\cA$ involves what $\Phi^\cB$ does on $\cA$.)

While this bound is natural within this context, it turns out not to be
so useful for our two basic classes of example. In 
Proposition \ref{alth} below we will give a different bound that
does turn out to be useful for our basic examples. But perhaps
other examples will arise for which the above bound is useful.

We now summarize the main points discussed in this section.

\begin{defn}
\label{defexp}
Let $\cA$ and $\cB$ be unital C*-algebras and 
let $\Pi=(\cD, \o)$ be a bridge from
$\cA$ to $\cB$. We say that $\Pi$ is a \emph{bridge with conditional
expectations} if conditional expectations $E^\cA$ and $E^\cB$ from
$\cD$ onto $\cA$ and $\cB$ are specified, satisfying the following 
properties:
\begin{enumerate}
\item The conditional expectations are faithful. 
\item The pivot $\o$ is a projection.
\item There is a constant, $r_\o$, such that
\[
E^\cA(\o) = r_\o1_\cD = E^\cB(\o)  .
\]
\end{enumerate} 
For such a bridge with conditional expectations we define
$\Phi^\cA$ on $\cD$ by
\[
\Phi^\cA(d) = r_\o^{-1} E^\cA(\o d\o)   .
\]
We define $\Phi^\cB$ similarly, with the 
roles of $\cA$ and $\cB$ reversed.
We will often write $\Pi = (\cD, \o, E^\cA, E^\cB)$ for a bridge
with conditional expectations.
\end{defn}

I should mention here that at present I do not see how the class of
examples considered by \Lat \ that involves 
non-commutative tori \cite{Ltr3}
fits into the setting of bridges with conditional expectations, though
I have not studied this matter carefully. It would certainly be
interesting to understand this better. I also do not see how the
general case of ordinary compact metric spaces, as discussed
in theorem 6.6 of \cite{Ltr2}, fits into the setting of 
bridges with conditional expectations

\begin{defn}
\label{admis}
With notation as above, let  $L^\cA$ and
$L^\cB$ be slip-norms on $\cA$ and $\cB$. 
We say that a bridge with conditional 
expectations $\Pi = (\cD, \o, E^\cA, E^\cB)$ 
is \emph{admissible for $L^\cA$ and
$L^\cB$} if
\[
L^\cA(\Phi^\cA(b)) \leq L^\cB(b)   
\]
for all $b \in \cB$, and
\[
L^\cB(\Phi^\cB(a)) \leq L^\cA(a)   
\]
for all $a \in \cA$  .
\end{defn}

We define the reach, height and length of a bridge with 
conditional expectations  $(\cD, \o, E^\cA, E^\cB)$ to
be those of the bridge  $(\cD, \o)$.

From the earlier discussion we obtain:

\begin{thm}
\label{thmadmis}
Let  $L^\cA$ and
$L^\cB$ be slip-norms on unital C*-algebras $\cA$ and $\cB$,
and let  $\Pi = (\cD, \o, E^\cA, E^\cB)$ be a bridge with conditional expectations from $\cA$ to $\cB$
that is admissible for $L^\cA$ and
$L^\cB$. Then 
\[
\mathrm{reach}(\Pi) \leq \max\{\g^\cA, \g^\cB\},
\]
where
\[
\g^\cA = \sup\{\|a\o - \o\Phi^\cB(a)\|_\cD : a \in \cL^1_\cA \},
\]
and similarly for $\g^\cB$, while
\[
\mathrm{height}(\Pi) \leq \max\{\d^\cA, \d^\cB\}  ,
\]
where
\[
\d^\cA = \sup \{\|a - \Phi^\cA(a)\|: a \in \cL^1_\cA\}
\]
and similarly for $\d^\cB$. Consequently
\[
\mathrm{length}(\Pi) \leq \max\{\g^\cA, \g^\cB, \d^\cA, \d^\cB\}  .
\]
(Consequently the propinquity between $(\cA, L^\cA)$ and $(\cB, L^\cB)$,
as defined in \cite{Ltr2}, is no greater than the right-hand side above.) 
\end{thm}

We could axiomitize the above situation in terms of just $\Phi^\cA$ and
$\Phi^\cB$, without requiring that they come from conditional expectations,
but at present I do not know of examples for which this would be
useful. It would not suffice to require that $\Phi^\cA$ and
$\Phi^\cB$ just be positive (and unital)  because for the matricial
case discussed in the next
section they would need to be completely positive.

The following result is very pertinent to our first
class of basic examples. 
\begin{prop}
\label{abel} With notation as above, suppose that our bridge
$\Pi$ has the quite special property that $\o$ commutes
with every element of $\cA$, or at least that 
$E^\cA(\o a\o) = E^\cA(a\o)$ for all $a \in \cA$. Then 
$\Phi^\cA(a) = a$ for all $a \in \cA$. Consequently
$\d^\cA = 0$, and the
restriction of $S_1(\o)$ to $\cA$ is all of $S(\cA)$.
\end{prop}

\begin{proof}
This depends on the conditional expectation property
of $E^\cA$. For  $a \in \cA$ we have
\[
\Phi^\cA(a) = r_\o^{-1}E^\cA(a\o) 
= ar_\o^{-1}E^\cA(\o) =   a  .
\]
\end{proof}

The following steps might not initially seem useful, but 
in Sections \ref{app1} and \ref{app2} we will
see in connection with our basic examples that they are quite 
useful. Our
notation is as above. Let $\nu \in S(\cB)$. Then
as seen above, $\nu \circ \Phi^\cB \in S_1(\o)$, and so its
restriction to $\cA$ is in $S(\cA)$. But then 
$\nu \circ \Phi^\cB \circ \Phi^\cA \in S_1(\o)$. Let us
denote it by $\psi_\nu$. Then the restriction of $\psi_\nu$
to $\cB$ can be used as an
approximation to $\nu$ by an element of $S_1(\o)$. Now
for any $b \in \cB$ we have
\[
|\nu(b) - \psi_\nu(b)| = 
|\nu(b) - (\nu \circ \Phi^\cB \circ \Phi^\cA)(b)| 
\leq \|b - \Phi^\cB(\Phi^\cA(b))\|  .
\]
\begin{notation}
\label{bert}
In terms of the above notation we set
\[
\hat\d^\cB = \sup \{\|b - \Phi^\cB(\Phi^\cA(b))\| : b \in \cL^1_\cB\}.
\]
\end{notation}
We note that $L^\cB( \Phi^\cB(\Phi^\cA(b))) \leq L^\cB(b)$ because
of the admissibility requirements of Definition \ref{admis}. It follows that
\[
\rho_{L^\cB}(\nu, \ \psi_\mu) \leq \hat \d^\cB  .
\]
We define $\hat \d^\cA$ in the same way, and obtain the
corresponding estimate for the distances from elements
of $S(\cA)$ to the restriction of $S_1(\o)$ to $\cA$. 
In this way we obtain:
\begin{prop}
\label{alth}
For notation as above,
\[
\mathrm{height}(\Pi) \leq 
\max\{\min\{ \d^\cA, \hat \d^\cA\}, \min\{\d^\cB, \hat \d^\cB\}\}  .
\]
\end{prop}
We will see in Section \ref{app1} that for our first class of
basic examples, $ \Phi^\cB \circ \Phi^\cA$ is exactly
a term that plays an important role 
in \cite{R7, R21}. It is essentially
an ``anti-Berezin-transform''.

%%%%%%%%%%%%%%%%%%%%%%%%%%%%%%%%%
 
\section{The corresponding matricial bridges}
\label{matricial}

Fix a positive integer $q$. We let $M_q$ denote the C*-algebra
of $q \times q$ matrices with complex entries. For any C*-algebra
$\cA$ we let $M_q(\cA)$ denote the C*-algebra of $q \times q$ matrices with entries in $\cA$. We often identify it in the evident way with the C*-algebra
$M_q \otimes \cA$.

Let $\cA$ and $\cB$ be unital C*-algebras, and let $\Pi = (\cD, \o)$ be 
a bridge from $\cA$ to $\cB$. Then $M_q(\cA)$ can be viewed as
a subalgebra of $M_q(\cD)$, as can $M_q(\cB)$. Let
$\o_q = 1_q \otimes \o$, where $1_q$ is the identity element
of $M_q$, so $\o_q$ can be viewed as the diagonal
matrix in $M_q(\cD)$ with $\o$ in each diagonal entry.
Then it is easily seen that $\Pi^q = (M_q(\cD), \o_q)$ is a bridge
from $M_q(\cA)$ to $M_q(\cB)$.

In order to measure the length of $\Pi^q$ we need slip-norms
$L^\cA_q$ and $L^\cB_q$ on  $M_q(\cA)$ and $M_q(\cB)$.  
It is reasonable to want these slip-norms to be coherent
in some sense as $q$ varies. The discussion that we will
give just after Theorem \ref{thmht} suggests that the
coherence requirement be that the sequences
$\{L^\cA_q\}$ and $\{L^\cB_q\}$ form ``matrix slipnorms''.
To explain what 
this means, for any positive integers $m$ and $n$ we let
$M_{mn}$ denote  the linear space of $m \times n$ 
matrices with complex entries, equipped with the norm obtained
by viewing such matrices as operators from the Hilbert
space $\bC^n$ to the Hilbert space $\bC^m$. We then note
that for any $A \in M_n(\cA)$, any $\a \in M_{mn}$,
and any $\b \in M_{nm}$ the
usual matrix product $\a A \b$ is in $M_m(\cA)$. The
following definition, for the case of Lip-norms, 
is given in definition 5.1 of \cite{Wuw2} (and see
also \cite{Wuw1, Wuw3, R21, ER}). 

\begin{defn}
\label{mtx}
A sequence $\{L^\cA_n\}$  
is a  \emph{matrix slip-norm} for $\cA$ if 
$L^\cA_n$ is a $*$-seminorm (with value $+\infty$ permitted)
on $M_n(\cA)$ 
for each integer $n \geq 1$, and this family of
seminorms has the following properties:
\begin{enumerate}
\item For any  $A \in M_n(\cA)$, any $\a \in M_{mn}$, and
any $\b \in M_{nm}$, we have
\[
L_m^\cA(\a A \b) \leq \|\a\|L_n^\cA(A)\|\b\|.
\]
\item For any  $A \in M_m(\cA)$ and any  $C \in M_n(\cA)$
we have
\[
L_{m+n}^\cA\left(
\begin{bmatrix}
A & 0 \\
0 & C
\end{bmatrix}
\right)
= \max(L_m^\cA(A), L_n^\cA(C))    .
\]   
\item $L_1^\cA$ is a slip-norm. 
\end{enumerate}
\end{defn}

We remark that the properties above imply that for $n \geq 2$ 
the null-space of $L_n^\cA$ contains all of $M_n$,
not just the scalar multiples of the identity. This is why 
our definition of slip-norms does not require that the
null-space is exactly the scalar multiples of the identity. 

Now let $\Pi = (\cD, \o, E^\cA, E^\cB)$ be a bridge with
conditional expectations. For any integer $q \geq 1$ set 
$E^\cA_q = \iota_q \otimes E^\cA$, where $\iota_q$
is the identity map from $M_q$ onto itself. Define
$E^\cB_q$ similarly. Then it is easily seen that $E^\cA_q$
and $E^\cB_q$ are faithful conditional expectations from
$M_q(\cD)$ onto its subalgebras $M_q(\cA)$ and $M_q(\cB)$
respectively. Furthermore, $E^\cA_q(\o^q)$ is the diagonal
matrix each diagonal entry of which is $E^\cA(\o) = r_\o 1_\cD$, and
from this we see that $E^\cA_q(\o_q) = r_\o 1_{M_q(\cA)}$.
Thus $r_{\o_q} = r_\o$.
It is also clear that $\o_q$ is a projection. Putting this
all together, we obtain:

\begin{prop}
\label{promatx}
Let $\Pi = (\cD, \o, E^\cA, E^\cB)$ be a bridge with
conditional expectations from $\cA$ to $\cB$. Then
\[
\Pi^q = (M_q(\cD), \o_q, E^\cA_q, E^\cB_q)
\]
is a bridge with conditional expectations from
$M_q(\cA)$ to $M_q(\cB)$. It has the same constant
$r_\o$ as does $\Pi$.
\end{prop}

We can then set $\Phi^\cA_q = \iota_q \otimes \Phi^\cA$,
and similarly for $\Phi^\cB_q$. Because $\Pi^q$
has the same constant
$r_\o$ as does $\Pi$, we see that for any $D \in M_q(\cD)$
we have
\[
\Phi^\cA_q(D) = r_\o^{-1} E^\cA_q(\o_q D\o_q)   .
\]

Suppose now that $\cA$ and $\cB$ have matrix slip-norms
$\{L^\cA_n\}$ and  $\{L^\cB_n\}$. 
We remark that a matrix slip-norm $\{L^\cA_n\}$ 
is in general not
at all determined by $L^\cA_1$. Thus a bridge that is
admissible for $L^\cA_1$ as in Definition \ref{admis}
need not relate well to the seminorms $L^\cA_n$ for higher $n$.

\begin{defn}
\label{matadm}
With notation as above, let  $\{L_n^\cA\}$ and
$\{L_n^\cB\}$ be matrix slip-norms 
on $\cA$ and $\cB$. We say that
a bridge with conditional expectations $\Pi = (\cD, \o, E^\cA, E^\cB)$ 
is \emph{admissible for $\{L_n^\cA\}$} and
$\{L_n^\cB\}$ if for all integers $n \geq 1$ the bridge $\Pi^n$
is admissible for $L_n^\cA$ and
$L_n^\cB$; that is,  for all integers $n \geq 1$
we have
\[
L_n^\cA(\Phi_n^\cA(B)) \leq L_n^\cB(B)   
\]
for all $B \in M_n(\cB)$, 
where
\[
\Phi_n^\cA(B) = r_\o^{-1} E_n^\cA(\o_nB\o_n)   ,
\]
and similarly with the roles of $\cA$ and $\cB$ reversed.
\end{defn}

We assume now that $\Pi = (\cD, \o, E^\cA, E^\cB)$ 
is admissible for $\{L_n^\cA\}$ and
$\{L_n^\cB\}$.
Since for a fixed integer $q$ the bridge $\Pi^q$ is 
admissible for $L_q^\cA$ and
$L_q^\cB$, the length of $\Pi^q$ is
defined. We now show how to obtain an upper bound for
the length of $\Pi^q$ in terms of the data used in
the previous section to get an upper bound on the
length of $\Pi$.

We consider first the reach of $\Pi^q$. Set, much as earlier,
\[
\cL_\cA^{1q} = \{A \in M_q(\cA):  
A = A^* \quad \mathrm{and} \quad L^\cA_q(A) \leq 1\}   ,
\]
and similarly for $\cL_\cB^{1q}$. 
Then the reach of $\Pi^q$ is defined to be
\[
\mathrm{Haus}_{M_q(\cD)}\{\cL_\cA^{1q}  \o_q \ , \     \o_q  \cL_\cB^{1q}\}   .
\]
Suppose that $B \in \cL_\cB^{1q}$. Then $\Phi^\cA_q(B) \in  \cL_\cA^{1q}$
by the admissibility requirement. So we want to bound
\[
\|\Phi^\cA_q(B)\o_q - \o_q B\|_{M_q(\cD)}  .
\]
I don't see any better way to bound this in terms of the data used
in Theorem \ref{thmadmis}  for $\Pi$ than by using an entry-wise estimate 
as done in the third paragraph before
lemma 14.2 of \cite{R21}). We use the fact that 
for a $q \times q$ matrix 
$C = [c_{jk}]$ with entries in a C*-algebra 
we have $\|C\| \leq q \max_{jk}\{\|c_{jk}\|\}$ (as is
seen by expressing $C$ as the sum of the $q$ matrices whose 
only non-zero entries are the entries $c_{jk}$ for which $j-k$
is a given constant mod $q$). In this way, 
for $B \in M_q(\cB)$ with $B = [b_{jk}]$ we find that the last
displayed term above is
\[
\leq q \max_{jk}\{\|\Phi^\cA(b_{jk})\o - \o b_{jk}\|\}   .
\]
The small difficulty is that the $b_{jk}$'s need not be 
self-adjoint. But for any $b \in \cB$, if we denote its
real and imaginary parts by $b_r$ and $b_i$, then
because $L^\cB$ is a $*$-seminorm it follows that
$L^\cB(b_r) \leq L^\cB(b)$ and similarly for $b_i$.
Consequently
\begin{eqnarray*}
\|\Phi^\cA(b)\o - \o b\| 
\leq \|\Phi^\cA(b_r)\o - \o b_r\| +   \|\Phi^\cA(b_i)\o - \o b_i\|    \\
\leq \g^\cB L^\cB(b_r) + \g^\cB L^\cB(b_i) \leq 2\g^\cB L^\cB(b).
\end{eqnarray*}
Thus the term displayed just before is
\[
\leq 2q\g^\cB L^\cB(b_{jk})  .
\]
But
$\{L^\cB_n\}$ is a matrix slip-norm, and by the
first property of such seminorms given in Definition \ref{mtx}, we have
\[
 \max \{L^\cB(b_{jk})\}  \leq L^\cB_q(B) . 
\]
Thus for $B \in \cL_\cB^{1q}$ we see that
\[
\|\Phi^\cA_q(B)\o_q - \o_q B\|_{M_q(\cD)} \leq 2q\g^\cB,  
\]
so that $\o_q B$ is in the $2q \g^\cB$-neighborhood of 
$ \cL_\cA^{1q} \o_q$. In the same way
$A\o_q $ is in the $2q \g^\cA$-neighborhood of 
$\o_q \cL_\cB^{1q}$ for every $A \in \cL_\cA^{1q}$.
We find in this way that
\[
\mathrm{reach}(\Pi^q) \ \leq \ 2q \max \{ \g^\cA, \g^\cB\}  .
\]

We now consider the height of $\Pi^q$. We argue
much as in the discussion of height before Definition \ref{defexp}. 
For any $\mu \in S(M_q(\cA))$
its composition with $\Phi^\cA_q$ is an element, $\phi_\mu$, of $S_1(\o_q)$,
specifically defined by
\[
\phi_\mu(D) = \mu(\Phi^\cA_q(D)) = \mu(r_\o^{-1} E^\cA_q(\o_q D\o_q))   .
\]
We take $\phi_\mu |_{M_q(\cA)}$ as an approximation to $\mu$, 
and estimate
the distance between these elements of $S(M_q(\cA))$. 
For $A \in \cL_\cA^{1q}$ we calculate
\[
|\mu(A) - \phi_\mu(A)| = |\mu(A) - \mu(\Phi^\cA_q(A))|
\leq \|A - \Phi^\cA_q(A)\|  . 
\]
Again I don't see any better way to bound this in terms of the data used
in Theorem \ref{thmadmis} for $\Pi$ than by using an entry-wise estimates.  
For $A \in M_q(\cA)$ with $A = [a_{jk}]$ we find (by using arguments
as above to deal with the fact that the $a_{jk}$'s need
not be self-adjoint) that the last
displayed term above is
\[
\leq q \max_{jk}\{\|a_{jk} - \Phi^\cA(a_{jk})\|\}
\leq 2q \d^\cA \max \{L^\cA(a_{jk})\} .
\]
But again
$\{L^\cA_n\}$ is a matrix slip-norm, and so by the
first property of such seminorms given in Definition \ref{mtx} we have
\[
 \max \{L^\cA(a_{jk})\}  \leq L^\cA_q(A) . 
\]
Since our assumption is that $A \in \cL_\cA^{1q}$,
we see in this way that
\[
\rho_{L^\cA_q}(\mu, \ \phi_\mu |_\cA) \leq 2q\d^\cA  .
\]
Thus $S(M_q(\cA))$ is in the $2q\d^\cA$-neighborhood of the restriction
to $M_q(\cA)$ of $S_1(\o_q)$.

We find in the same way that $S(M_q(\cB))$ is in the $2q\d^\cB$-neighborhood of the restriction
to $M_q(\cB)$ of $S_1(\o_q)$. Consequently,
\[
\mathrm{height}(\Pi^q) \ \leq \ 2q \max \{ \d^\cA, \d^\cB\}  .
\]

We can instead use $\hat \d^\cA$ in the way done in Proposition
\ref{alth}. Using reasoning much like that used above, we find
that for any $B \in M_q(\cB)$ we have:

\begin{eqnarray*}
\|B-\Phi^\cB_q(\Phi^\cA_q(B))\| \leq q \max_{jk}\{\|b_{jk} - \Phi^\cB(\Phi^\cA(b_{jk}))\|  \\
\leq 2q \hat\d^\cB \max\{L^\cB(b_{jk})\} \leq 2q\hat\d^\cB L^\cB_q(B)
\end{eqnarray*}
Consequently we see that
\[
\mathrm{height}(\Pi^q) \leq 2q\max\{\hat\d^\cA, \hat\d^\cB\}  .
\]

We summarize what we have found by:

\begin{thm}
\label{thmmatad}
Let  $\{L^\cA_n\}$ and
$\{L^\cB_n\}$ be matrix slip-norms 
on unital C*-algebras $\cA$ and $\cB$,
and let  $\Pi = (\cD, \o, E^\cA, E^\cB)$ be a bridge with conditional expectations 
from $\cA$ to $\cB$ that is admissible for  $\{L^\cA_n\}$ and
$\{L^\cB_n\}$. For any fixed positive integer $q$ let $\Pi^q$ be the
corresponding bridge with conditional expectations 
from $M_q(\cA)$ to $M_q(\cB)$. Then 
\[
\mathrm{reach}(\Pi^q) \leq 2q\max\{\g^\cA, \g^\cB\},
\]
where as before
\[
\g^\cA = 
\sup\{\|a\o - \o\Phi^\cB(a)\|_\cD : a \in \cL^1_\cA \},
\]
and similarly for $\g^\cB$; while
\[
\mathrm{height}(\Pi^q) \leq 
2q\max\{ \min\{\d^\cA, \hat\d^\cA\},  \min\{\d^\cB\, \hat\d^\cB\}\}  ,
\]
where as before
\[
\d^\cA = \sup \{\|\Phi^\cA(a) - a\|: a \in \cL^1_\cA\}
\]
and  
\[
\hat\d^\cA = \sup \{\|a - \Phi^\cA(\Phi^\cB(a))\| : a \in \cL^1_\cA\},
\]
and similarly for $\d^\cB$ and $\hat \d^\cB$. 
Consequently
\[
\mathrm{length}(\Pi^q) \leq 
2q\max\{\g^\cA, \g^\cB, \min\{\d^\cA, \hat\d^\cA\},  \min\{\d^\cB\, \hat\d^\cB\}\}  .
\]
\end{thm}

%%%%%%%%%%%%%%%%%%%%%%%%%%%%%%%%%%%%%%%%

\section{The application to the first class of basic examples}
\label{app1}

We now apply the above general considerations to our first
class of basic examples, described in Section \ref{basic1}.
Thus we have $G$, $\cH$, $\cB$, $P$, $H$, $\cA$, $\cD$
and $\o$ as defined there, as well as $L^\cA$ and $L^\cB$
given by equation \eqref{lipn}.
We proceed to obtain an upper bound for the length of
the bridge $\Pi = (\cD, \o)$, where $\cD = C(G/H, \cB)$
and $\o(x) = \a_x(P)$. We begin by considering
its reach.

As seen in Section \ref{brce} , 
for any $F \in \cD = C(G/H, \cB)$ we have  
\[
E^\cA(\o F\o)(x) = \t_\cB(F(x)\a_x(P))  .
\]
From this it is easily seen that $r_\o^{-1}$
is the dimension of $\cH$, and so $r_\o^{-1}\t_\cB$
is the usual unnormalized trace on $\cB$, which
we now denote by $\rtr_\cB$.
In particular, for any $T \in \cB$ we have
\begin{equation}
\Phi^\cA(T)(x) = r_\o^{-1}E^\cA(\o T\o)(x) = \mathrm{tr}_\cB(T\a_x(P))
\label{phia}
\end{equation}
for all $x \in G/H$. But this is exactly the covariant Berezin symbol of $T$ (for this general context) 
as defined early in section 1 of \cite{R7} 
and denoted there by $\s_T$. 
It is natural to put on $\cD$ the action $\l\otimes \a$ of $G$. One
then easily checks that $\Phi^\cA$ is equivariant for
$\l\otimes \a$ and $\l$. From this
it is easy to verify that 
\[
L^\cA(\Phi^\cA(T)) = L^\cA(\s_T) \leq L^\cB(T)
\]
for all $T \in \cB$, which is exactly the content of
proposition 1.1 of \cite{R7}. Thus that part of
admissibility is satisfied.

Now
\[
(\Phi^\cA(T)\o - \o T)(x) = \a_x(P)(\s_T(x)1_\cB - T).
\]
Consequently
\begin{equation*}
\|\Phi^\cA(T)\o - \o T\|_\cD = \sup\{\|\a_x(P)(\s_T(x)1_\cB - T)\|_\cB: x \in G/H\}  .
%\label{phito}
\end{equation*}
As discussed in the text before proposition 8.2 of \cite{R21}, by
equivariance this is
\[
= \sup\{\|P(\s_{\a_x(T)}1_\cB - \a_x(T))\|_\cB : x \in G\}  .
\]
Then because $\a_x$ is isometric on $\cB$ for $L^\cB$ 
(as well as for the norm), 
we find for our present example 
that $\g^\cB$, as defined in equation \eqref{gamb1}, is given by
\begin{IEEEeqnarray}{rCl}
\label{gamb}
\g^\cB & = & \sup\{\|\Phi^\cA(T)\o - \o T\|_\cD : T \in \cL^1_\cB \}   \\
& = & \sup\{\|P(\mathrm{tr}(PT) 1_\cB - T)\|_\cB : T \in \cL^1_\cB \}  .
\nonumber 
\end{IEEEeqnarray}
This last term is exactly 
the definition of $\g^\cB$ given
in proposition 8.2 of \cite{R21}.

We next consider $\g^\cA$. For any $f \in \cA$ we have
\begin{equation}
\Phi^\cB(f) = r_\o^{-1}E^\cB(f)      
= d_\cH\int_{G/H} f(x)\a_x(P) \ dx  .
\label{phib}
\end{equation}
where $d_\cH = \mathrm{dim}(\cH)$.
But this is exactly the formula used for the Berezin contravariant
symbol, as indicated in Section \ref{brce}. Early in section 2 of \cite{R7}  this $\Phi^\cB$ is denoted by
$\breve \s _f$, that is, 
\begin{equation}
\label{brsig}
\Phi^\cB(f) = \breve \s_f   
\end{equation}
for the present class of examples. One
easily checks that $\Phi^\cB$ is equivariant for
$\l\otimes \a$ and $\a$. From this
it is easy to verify that 
\[
L^\cB(\Phi^\cB(f)) = L^\cB(\breve \s_f) \leq L^\cA(f)  
\]
for all $f \in \cA$, as shown in section 2 of \cite{R7}.
Thus we obtain:
\begin{prop}
\label{adm1}
The bridge with conditional expectations \\
$\Pi = (\cD, \o, E^\cA, E^\cB)$ is admissible for
$L^\cA$ and $L^\cB$.
\end{prop}
Much as in the statement of proposition 8.1 of \cite{R21} set 
\begin{equation}
\label{gama}
\breve\g^\cA = d_\cH \int \rho_{G/H}(e, y)\|P\a_y(P)\| dy  .
\end{equation}
In the proof of proposition 8.1 of \cite{R21} 
(given just before the statement of the proposition,
and where $\breve\g^\cA$ is denoted
just by $\g^\cA$) it is shown, with different notation,  
that
\begin{equation}
\|f\o - \o \breve \s_f\| \leq \breve\g^\cA L^\cA(f) .
\label{gamin}
\end{equation}
Thus if $f \in \cL^1_\cA$ then
\[
\|f\o - \o \Phi^\cB(f)\| \leq \breve\g^\cA .
\]
It follows that $\g^\cA  \leq \breve\g^\cA $. We
have thus obtained:

\begin{prop}
\label{prore1}
For the present class of examples, with notation as above,
we have
\[
\mathrm{reach}(\Pi) \leq  
\max\{\g^\cA, \g^\cB\} \leq  \max\{\breve\g^\cA, \g^\cB\}
\]
where $\breve\g^\cA$ is defined in equation \ref{gama}
and $\g^\cB$ is defined in equation \ref{gamb} 
(and \ref{gamb1}) above.
\end{prop}

Even more, for the case mentioned
at the beginning of Section \ref{basic2} in which $G$ is a 
compact semisimple Lie group and
$\l$ is
a positive integral weight, for each positive integer $m$ 
let $(\cH^m, U^m)$ be the irreducible representation of $G$
with highest weight $m\l$. Then let $\cB^m = \cL(\cH^m)$
with action $\a$ of $G$, and let $P^m$ be the projection on the
highest weight vector in $\cH^m$. All the $P^m$'s will have
the same $\a$-stability group, $H$. As before, we let
$\cA = C(G/H)$. Then for each $m$ we can construct
as in Section \ref{brce} the bridge with conditional 
expectations, $\Pi_m = (\cD_m, \o^m,  E^\cA_m, E^\cB_m)$.
From a fixed length function $\ell$ on $G$ we will obtain
Lip-norms $\{L^{\cB^m}\}$ which together
with $\{L^{\cA}\}$ give meaning to the lengths of the bridges
$\Pi_m$. In turn the constants $\g^\cA_m$, $\breve \g^\cA_m$, 
$\g^{\cB^m}$, $\d^\cA_m$, $\d^{\cB^m}$ will be defined.

Now it follows from the discussion of $\breve\g^\cA$
above that $\g^\cA_m \leq \breve\g^\cA_m$ for each $m$.
But section 10 of \cite{R21} gives a proof 
that the sequence
$\breve \g^\cA_m$ converges to 0 as $m$ goes to $\infty$.
It follows that  $\g^\cA_m$ 
converges to 0 as $m$ goes to $\infty$.
Then section 12 of \cite{R21} gives a proof that the sequence
$\g^{\cB^m}$ converges to 0 as $m$ goes to $\infty$.
Putting together these results for $\g^\cA_m$
and $\g^{\cB^m}$, we obtain:

\begin{prop}
\label{progam}
The reach of the bridge $\Pi_m$ goes to 0 as $m$ goes to $\infty$.
\end{prop} 

We now consider the height of $\Pi$. For $\d^\cA$ something
quite special happens. It is easily seen that $\cA = C(G/H)$ is
the center of $\cD = \cA \otimes \cB$, and so all
elements of $\cA$ commute with $\o$. Thus we can
apply Proposition \ref{abel} to conclude that
$\d^\cA = 0$.

In order to deal with $\cB$ we use $\hat \d^\cB$ of 
Notation \ref{bert} and the discussion surrounding it.
For any $T \in \cB$ we have
%\begin{eqnarray*}
\begin{align*}
\Phi^\cB(\Phi^\cA(T)) &= r_\o^{-1}E^\cB(\o (r_\o^{-1}E^\cA(\o T \o)) \o)  \\
&= d_\cH \int \a_x(P) (d_\cH \t_\cB(\a_x(P)T\a_x(P))\a_x(P) dx  \\
&= d_\cH \int \a_x(P) (tr_\cB(\a_x(P)T) dx = \breve \s(\s_T)  .
\end{align*}
%\end{eqnarray*}
where for the last term we use notation
from \cite{R7,R21}. The term $\breve \s(\s_T)$ plays an important
role there. See  theorem 6.1 of \cite{R7} 
and theorem 11.5 of \cite{R21}. 
The $\hat \d^\cB$ of our Notation \ref{bert} is for the present
class of examples exactly the $\d^\cB$
of notation 8.4 of \cite{R21}. For use in the next section we
here denote it by $\tilde \d^\cB$, that is:
\begin{notation}
\label{tdel}
For $G$, $\cA = C(G/H)$, $\cB = \cL(\cH)$, 
$\s$, $\breve \s$, etc as above,
we set
\[
\tilde \d^\cB = \sup \{\|T - \breve \s(\s_T)\| : T \in \cL^1_\cB\}.
\]
\end{notation}
When we combine this with Propositions \ref{alth} and \ref{prore1} we
obtain:
\begin{thm}
\label{thmprod}
For the present class of examples, 
with notation as above, we have
\[
\mathrm{height}(\Pi) \leq \tilde \d^\cB   .
\]
Consequently
\[
\mathrm{length}(\Pi)  \leq \max\{\g^\cA, \g^\cB, \min\{\d^\cB, \tilde \d^\cB\} 
\leq \max\{\breve\g^\cA, \g^\cB, \min\{\d^\cB, \tilde \d^\cB\} \}  .
\]
\end{thm}

We will indicate in Section \ref{trek} \Lat's definition of his
propinquity between compact quantum metric spaces, but
it is always no larger than the length of any bridge between
the two spaces. He denotes his propinquity simply by $\L$,
but we will denote it here by ``Prpq''. Consequently,
from the above theorem we obtain:
\begin{cor} 
\label{corpr}
With notation as above, 
\[
\mathrm{Prpq}((\cA, L^\cA), \ (\cB, L^\cB)) \leq  
\max\{\breve\g^\cA, \g^\cB, \min\{\d^\cB, \tilde \d^\cB\} \}  .
\]
\end{cor}

For
the case of highest weight representations discussed just
above, theorem 11.5 of \cite{R21} gives a proof 
that the sequence
$\tilde \d^\cB_m$ (in our notation) 
converges to 0 as $m$ goes to $\infty$.
It follows from the above proposition that:
\begin{prop}
\label{prodel}
The height of the bridge $\Pi_m$ goes to 0 as $m$ goes to $\infty$.
\end{prop} 

Combining this with Proposition \ref{progam}, we obtain:

\begin{thm}
\label{thmht}
The length of the bridge $\Pi_m$ goes to 0 as $m$ goes to $\infty$.
Consequently $\mathrm{Prpq}((\cA, L^\cA), \ (\cB^m, L^{\cB^m}))$
goes to 0 as $m$ goes to $\infty$.
\end{thm}

We now treat the matricial case, beginning with the general
situation in which $G$ is some compact group. We must
first specify our matrix slip-norms. This is essentially done
in example 3.2 of \cite{Wuw2} and section 14 of \cite{R21}.
As discussed in Section \ref{basic1}, we have the 
actions $\l$ and $\a$ on $\cA = C(G/H)$ and $\cB = \cB(\cH)$
respectively. For any $n$ let $\l^n$ and $\a^n$ be the
corresponding actions $\iota_n\otimes \l$ and $\iota_n \otimes \a$
on $M_n\otimes \cA = M_n(\cA)$ and $M_n\otimes \cB = M_n(\cB)$.
We then use the length function $\ell$ and formula \ref{lipn} to
define seminorms $L^\cA_n$ and $L^\cB_n$ on $M_n(\cA)$
and $M_n(\cB)$. It is easily verified that $\{L^\cA_n\}$ and
$\{L^\cB_n\}$ are matrix slip-norms. Notice that here $L_1^\cA = L^\cA$
and $L_1^\cB = L^\cB$ are actually lipnorms, and so, by property
1 of Definition \ref{mtx}, for each $n$ the null-spaces of $L^\cA_n$ and
$L^\cB_n$ are exactly $M_n$.

Now fix $q$ and take $n=q$.  From our bridge with conditional
expectations
$\Pi = (\cD, \o, E^\cA, E^\cB)$ we define the bridge
with conditional expectations
$\Pi^q = (M_q(\cD), \o_q, E_q^\cA, E_q^\cB)$ 
between $M_q(\cA)$ and $M_q(\cB)$ in the
way done in Proposition \ref{promatx}. We then define
$\Phi^\cA_q = \iota_q \otimes \Phi^\cA$,
and similarly for $\Phi^\cB_q$ as done right after
Proposition \ref{promatx}.

Because $\l$ and $\a$ and $\Phi^\cA$ and $\Phi^\cB$
act entry-wise on $M_q(\cD)$, and because $\Phi^\cA$ and $\Phi^\cB$
are equivariant for $\l\otimes\a$ and $\l$ and for $\l\otimes\a$ and  $\a$ 
respectively, it is easily seen that
$\Pi^q$ is admissible for $\{L^\cA_q\}$ and
$\{L^\cB_q\}$.  

We are thus in position to apply Theorem \ref{thmmatad}.
From it and Theorem \ref{thmprod} we conclude that:

\begin{thm}
\label{thmmat1}
With notation as above, we have
\[
\mathrm{reach}(\Pi^q) \leq 2q\max\{\g^\cA, \g^\cB\}
\leq 2q\max\{\breve\g^\cA, \g^\cB\},
\]
where $\breve\g^\cA$ is defined by formula \ref{gama}.
Furthermore
\[
\mathrm{height}(\Pi^q) \leq 2q \min\{\d^\cB, \tilde \d^\cB\},
\]
where $\tilde \d^\cB$ is defined in Notation \ref{tdel}.
Thus
\[
\mathrm{length}(\Pi^q) \leq 2q 
\max\{\breve\g^\cA, \g^\cB,  \min\{\d^\cB, \tilde \d^\cB\} \} .
\]
\end{thm}

We remark that we could improve slightly on the above theorem by
using a calculation given in section 14 of \cite{R21} 
in the middle of the discussion there of Wu's results.
Let $F \in M_q(\cA)$ be given, with $F = \{f_{jk}\}$,
and set 
\[
t_{jk} = \breve \s_{f_{jk}} = d_\cH \int f_{jk}(y)\a_y(P) dy  ,
\]
and let $T = \{t_{jk}\}$. Then
\begin{eqnarray*}
(F\o_q - \o_q T)(x)  = \{\a_x(P)(f_{jk}(x) - d_\cH \int f_{jk}(y)\a_y(P) dy)\}  \\
=\{d_\cH \int (f_{jk}(x) - f_{jk}(y))\a_x(P)\a_y(P) dy\}  .
\end{eqnarray*}
To obtain a bound on $\g^\cA_q$ we need to take the supremum
of the norm of this expression over all $x$ and over all $F$
with $L^\cA_q(F) \leq 1$. By translation by $x$, in the way done
shortly before proposition 8.1 of \cite{R21}, it suffices to
consider 
\begin{align*}
\sup\{ \| \{d_\cH & \int (f_{jk}(e) - f_{jk}(y))P\a_y(P) dy\} \| \}   \\
&\leq d_\cH \int \| F(e) - F(y)\| \|P\a_y(P)\| dy   \\
&\leq L^\cA_q(F) d_\cH \int \rho_{G/H}(e, y) \| P \a_y(P)\| dy 
= L^\cA_q(F) \breve\g^\cA  .
\end{align*}
In this way we see that $\g^\cA_q \leq \breve \g^\cA$, with
no factor of $2q$ needed.

We can apply Theorem \ref{thmmat1} 
to the situation considered before
Proposition \ref{progam} in which $G$ is a 
compact semisimple Lie group, $\l$ is
a positive integral weight, and 
$(\cH^m, U^m)$ is the irreducible representation of $G$
with highest weight $m\l$ for each positive $m$,
with $\cB^m = \cL(\cH^m)$. We can then form the bridge with conditional 
expectations, $\Pi_m = (\cD_m, \o^m,  E^\cA_m, E^\cB_m)$ that
is discussed there. For any positive integer $q$ we then have
the matricial version involving $M_q(\cA)$, $M_q(\cB_m)$,
and the corresponding bridge $\Pi^q_m$. On applying
Theorem \ref{thmmat1} together with the results mentioned
above about the convergence of the quantities 
$\breve\g^\cA_m, \g^\cB_m$, and $\tilde \d^\cB_m$ 
to 0, we obtain one of the two main
theorems of this paper:

\begin{thm}
\label{thmmat1q}
With notation as above, we have
\[
\mathrm{length}(\Pi^q_m) \leq 2q 
\max\{\breve\g^\cA_m, \g^\cB_m,  \min\{\d^\cB_m, \tilde \d^\cB_m\} \} 
\]
where $\breve\g^\cA_m$ is defined as in formula \ref{gama}, and
where $\tilde \d^\cB_m$ is defined as in Notation \ref{tdel}.
Consequently $\mathrm{length}(\Pi^q_m)$ converges
to 0 as $m$ goes to $\infty$, for each fixed $q$. 
\end{thm}

We remark that because of the factor $q$ in the right-hand
side of the above bound for $\mathrm{length}(\Pi^q_m)$,
we do not obtain convergence to 0 that is uniform in $q$.
I do not have a counter-example to the convergence being
uniform in $q$, but it seems to me very possible that the convergence
will not be uniform.

%%%%%%%%%%%%%%%%%%%%

\section{The application to the second class of basic examples}
\label{app2}

We now apply our general considerations to our second
basic class of examples, described in Section \ref{basic2}.
We use the notation of that section.
We also use much of the notation of Section \ref{app1}, but now
we have  two representations, $(\cH^m, U^m)$ and $(\cH^n, U^n)$
(where for the moment $m$ and $n$ are just labels). We have
corresponding C*-algebras $\cB^m$ and $\cB^n$, 
and projections $P^m$ and $P^n$.

We let $L^m$ be the Lip-norm defined on $\cB^m$ determined 
by the action $\a$ and the length function $\ell$ as in 
equation \ref{lipn}, and similarly for $L^n$ on $\cB^n$. In
terms of these Lip-norms the length of any bridge from
$\cB^m$ to $\cB^n$ is defined.

As in Section \ref{basic2} we consider the bridge $\Pi = (\cD, \o)$  for which
\[
\cD = \cB^m \otimes \cA \otimes \cB^n \ = \ C(G/H, \ \cB^m \otimes \cB^n)  ,
\]
and  the pivot, $\o$, in $C(G/H, \ \cB^m \otimes \cB^n)$, is defined by
\[
\o(x) = \a_x(P^m) \otimes \a_x(P^n)   .
\] 
We view $\cB^m$ as a subalgebra of $\cD$ by sending
$T \in \cB^m$ to $T \otimes 1_\cA \otimes 1_{\cB^n}$, and similarly for 
$\cB^n$. 

As seen in Section \ref{brce},
the tracial state $\t_\cA \otimes \t_n$ on $\cA \otimes \cB^n$
determines a canonical conditional
expectation, $E^m$, from $\cD$ onto $\cB^m$,
and the tracial state $\t_m \otimes \t_\cA$ determines
a canonical conditional expectation, $E^n$, from $\cD$ onto
$\cB^n$ .
We find that for any $F \in \cD$ we have
\[
E^m(F)   = \int_{G/H} (\iota_m 
\otimes \t_n)(F(x)) \ dx  ,
\]
and similarly for $E^n$, where here $\iota_m$ is the
identity map from $\cB^m$ to itself.
From this it is easily seen that 
\[
r_\o^{-1} = d_md_n
\]
where $d_m$ is the dimension of $\cH^m$ and similarly for $d_n$.
Thus $\Pi = \{\cD, \o, E^m, E^n\}$ is a bridge with conditional
expectations.

Then
\[
\Phi^m(F)(x)   
= r_\o^{-1}E^m(\o F\o)(x) = d_m d_n E^m(\o F\o)(x)  .
\]
But, if we set $\a_x(P^m \otimes P^n) = \a_x(P^m) \otimes \a_x(P^n)$, we have
\[
E^m(\o F\o)(x)   
= \int (\iota_m \otimes \t_n)(\a_x(P^m \otimes P^n)F(x)\a_x(P^m \otimes P^n))  dx  .
\]
In particular, for any $T \in \cB^n$ we have
\[
E^m(\o T\o)(x) = \int \a_x(P^m)\t_n(T\a_x(P^n)) dx ,
\]
and so since $d_n \t_{\cB^n}$
is the usual unnormalized trace $\rtr_n$ on $\cB^n$, we have
\[
\Phi^m(T) = d_m \int \a_x(P^m) \rtr_n(T\a_x(P^n)) dx  .
\]
This is essentially the formula obtained in Section \ref{brce}, and 
is exactly the second displayed formula in section 3 of \cite{R25}. 
Even more, with notation as in
Section \ref{app1}, especially the $\Phi^\cA$ of equation \eqref{phia},
 except for our different $\cB^m$
and $\cB^n$ etc, we see that we can write
\begin{equation}
\label{dphi}
\Phi^m(T) = \Phi^{\cB^m}(\Phi^\cA(T)) = \breve \s^m(\s^n_T)   .
\end{equation}
We have a similar equation for $\Phi^n(T)$, and we see
that we depend on the context to make clear on which
of the two algebras $\cB^m$ and
$\cB^n$ we consider $\Phi^\cA$ to be defined. 

As in the proof of Proposition \ref{adm1}, we can use
the fact that $E^m$ and $E^n$ are equivariant, where the
action of $G$ on $\cD$ is given by $\a \otimes \l \otimes \a$, 
to obtain:
\begin{prop}
\label{adm2}
The bridge with
conditional expectations $\Pi$ is admissible for $L^m$ and $L^n$.
\end{prop}
The formula \eqref{dphi} suggests the following steps for obtaining a
bound on the reach of $\Pi$ in terms of the data of
the previous section. Let $S \in \cB^m$,
$f \in C(G/H)$, and $T \in \cB^n$. Then, for the
norm of $\cB^m \otimes \cB^m$ and for any $x \in G/H$, we have
\begin{align}
%eqnarray*}
\|(S \o - \o & T)(x)\| 
 =\|(S\a_x(P^m))\otimes \a_x(P^n) - 
\a_x(P^m)\otimes (\a_x(P^n)T)\|   \nonumber  \\ 
& \leq \|(S\a_x(P^m))\otimes \a_x(P^n) - 
f(x)(\a_x(P^m)\otimes (\a_x(P^n))\|   \nonumber  \\
& \quad + \|f(x)(\a_x(P^m)\otimes (\a_x(P^n)) - 
\a_x(P^m)\otimes (\a_x(P^n)T) \|  \nonumber \\
& = \|S\a_x(P^m) - f(x)\a_x(P^m)\| +
 \|f(x)\a_x(P^n) - \a_x(P^n)T \|   .
 \label{bigineq}
\end{align}
%eqnarray*
Notice that the last two norms are in $\cB^m$ and
$\cB^n$ respectively.

We will also use the $\Phi^\cB$ of equation
\eqref{phib}, but now, to distinguish it from the $\Phi^m$
above, we indicate that it is defined on $\cA$ (and maps to $\cB^m$) by
writing $\Phi_\cA^{\cB^m}$.

For fixed $T \in \cB^n$ let us set
$f(x) = \Phi^\cA(T) = \rtr_n(\a_x(P^n)T)$, and then
let us set $S = \Phi_\cA^{\cB^m}(f) = d_m\int f(x)\a_x(P^m)$. 
Thus $S = \breve \s^m_f$ by equation \ref{brsig}.
When we substitute these into the inequality \eqref{bigineq},
we obtain
\begin{align*}
&\|(\Phi_\cA^{\cB^m}(f)\o - \o T)(x)\| \\ &\leq 
\|(\Phi_\cA^{\cB^m}(f)(x) - f(x))\a_x(P^m)\| +
 \|\a_x(P^n)(f(x) -T) \| 
\end{align*}
In view of the definition of $f$, we recognize
that the supremum over $x \in G/H$ of
the second term on the right of the inequality sign is
the kind of term involved in the supremum in the
right-hand side of equality \eqref{gamb}.
Consequently that second term above is no greater than
$\g^{\cB^n}L^n(T)$. 
To indicate that this comes from
equality \eqref{gamb} we write $\g^{\cB^n}_\cA$ instead
of just $\g^{\cB^n}$.
Because $\Phi_\cA^{\cB^m}(f) = \breve \s^m_f$,
we also recognize that the supremum over $x \in G/H$ of
the first term above on the right of the inequality sign is
exactly 
(after taking adjoints to get $P^m$ on the correct side)
 the left hand side of inequality \eqref{gamin},
where the $\o$ there is that of Section \ref{app1}.
Consequently that term is no greater than 
$\breve \g^\cA_m L^\cA(f)$, where the subscript
$m$ on $\breve \g^\cA_m$ indicates that
$P^m$ should be used in equation \eqref{gama}. 
But from the admissibility in Proposition \ref{adm1}
involving $\Phi_\cA^{\cB^m} = \Phi^{\cB^m}$ and $\Phi^\cA$ we have 
\[
L^m(S) = L^m(\Phi_\cA^{\cB^m}(f)) \leq L^\cA(f) \leq L^n(T) .
\]
Notice that it follows that if $T \in \cL^1_{\cB^n}$ then 
 $S \in \cL^1_{\cB^m}$.
Anyway, on taking the supremum over $x \in G/H$, we obtain
\[
\|\Phi_\cA^{\cB^m}(f)\o - \o T\|
\leq (\breve \g^\cA_m  \ + \ \g^{\cB^n}_\cA) L^n(T)  .
\] 
 We see in this way
that the distance from $\o T$ to $\cL^1_{\cB^m} \o$
    is no bigger than  $\breve \g^\cA_m +  \g^{\cB^n}_\cA$. 

The role of $\cA$ in Theorem \ref{thmadmis} is here being 
played by $\cB^m$. So to reduce confusion we will here write
$\g^m_n$ for the $\g^\cA$ of Theorem \ref{thmadmis}, showing
also the dependence on $n$. Thus by definition
\[
\g^m_n = \sup\{\|T\o - \o \Phi^n(T)\|: T \in \cL^1_{B^m}\}.
\]
We define $\g^n_m$ similarly. Then in terms of this notation,
what we have found above is that 
\[
\g^n_m \leq \breve \g^\cA_m  + \g^{\cB^n}_\cA
\]
We now indicate the dependence of $\Pi$ on $m$ and
$n$ by writing $\Pi_{m,n}$.
The situation just above
is essentially symmetric in $m$ and $n$, 
and so, on combining this with the first inequality
of Theorem \ref{thmadmis}, we obtain:
\begin{prop}
\label{prorch2}
With notation as above, we have
\[
\mathrm{reach}(\Pi_{m,n}) \leq \max\{\g^n_m, \g^m_n\} \leq
\max\{\breve \g^\cA_m +  
\g^{\cB^n}_\cA, \breve \g^\cA_n +  \g^{\cB^m}_\cA\},
\]
\end{prop}
As mentioned in the previous section, 
section 10 of \cite{R21} gives a proof 
that the sequence
$\breve \g^\cA_m$ converges to 0 as $m$ goes to $\infty$,
while section 12 of \cite{R21} gives a proof that the sequence
$\g^{\cB^m}_\cA$ converges to 0 as $m$ goes to $\infty$.
We thus see that we obtain: 

\begin{prop}
\label{progamn}
The reach of the bridge $\Pi_{m,n}$ goes to 0 as $m$ 
and $n$ go to $\infty$ simultaneously.
\end{prop} 

We now obtain an upper bound for the height of $\Pi_{m,n}$.
For this we will again use Proposition \ref{alth}. We
calculate as follows, using equation \eqref{dphi}. 
For $T \in \cB^n$ we have
\begin{align*}
\Phi^n(\Phi^m(T)) = \Phi^n(\breve \s^m(\s^n_T)) = 
\breve \s^n(\s^m(\breve \s^m(\s^n_T)))
\end{align*}
Thus
\begin{align}
\label{2term}
\|T &- \Phi^n(\Phi^m(T))\|  \\
&\leq  \|T - \breve \s^n(\s^n_T)\|
+ \|\breve \s^n(\s^n_T) - 
\breve \s^n ((\s^m\circ \breve \s^m)(\s^n_T)\| \nonumber  \\
&\leq \tilde \d^{\cB^n}_\cA L^{\cB^n}(T)
+\|\s^n_T - \s^m(\tilde \s^m(\s^n_T))\|   ,  \nonumber
\end{align}
where the first term of the last line comes from 
Notation \ref{tdel} and we write $\tilde \d^{\cB^n}_\cA$
for the $ \tilde \d^{\cB^n}$ there.
But $\s^n_T$ is just an element of $\cA$, and
in inequality 11.2 of \cite{R21} it is shown that
for any $f \in \cA$ we have
\[
\|f - \s^m(\breve \s^m(f)\| \leq \tilde\d^\cA_m L^\cA(f),
\]
where $\tilde\d^\cA_m$ is defined in equation 11.1 of \cite{R21}
by
\begin{equation}
\label{dela}
\tilde\d^\cA_m = 
\int_{G/H} \rho_{G/H}(e, x)d_m \rtr(P^m \a_x(P^m))\ dx  .
\end{equation}
(In equation 11.1 of \cite{R21} $\tilde\d^\cA_m$ is denoted
just by  $\d^\cA_m$. Also, $ \s^m \circ\breve \s^m$ is,
within our setting, the usual Berezin transform.)
Thus we see that the second
term of the last line of inequality \eqref{2term} is no bigger 
than $\tilde \d^\cA_m L^\cA(\s^n_T)$.
But $L^\cA(\s^n_T) \leq L^{\cB^n}(T)$.
From all of this we see that if $T \in \cL^1_{\cB^n}$ then
\[
\|T - \Phi^n(\Phi^m(T))\| \leq \tilde \d^{\cB^n}_\cA + \tilde\d^\cA_m  .
\]
Again, the role of $\cB$ in Notation \ref{bert} is being played
here by $\cB^n$, and so to reduce confusion we will here
write $\hat \d^n_m$ for the $\hat \d^\cB$ of Notation \ref{bert}.
We then see that for for our
present class of examples, that depend on $m$ and $n$, 
we have
\[
\hat \d^n_m \leq  \tilde \d^{\cB^n}_\cA + \tilde\d^\cA_m .
\]
The situation is essentially symmetric in $m$ and $n$, and so,
combining this with Propositions \ref{alth} and \ref{prorch2},
we obtain:
\begin{thm}
\label{thmht2}
With notation as above, we have
\[
\mathrm{height}(\Pi_{m,n}) \leq \max\{\hat \d^n_m, \hat \d^m_n\} 
\leq \max\{\tilde \d^{\cB^n}_\cA + \tilde\d^\cA_m, \
\tilde \d^{\cB^m}_\cA + \tilde\d^\cA_n \}   .
\]
Consequently
\[
\mathrm{length}(\Pi_{m,n}) \leq \max\{\breve \g^\cA_m +  
\g^{\cB^n}_\cA, \ \breve \g^\cA_n +  \g^{\cB^m}_\cA, \ 
\tilde \d^{\cB^n}_\cA + \tilde\d^\cA_m, \
\tilde \d^{\cB^m}_\cA + \tilde\d^\cA_n \}  .
\]
\end{thm}
As mentioned in the previous section,
theorem 11.5 of \cite{R21} gives a proof 
that the sequence
$\tilde \d^{\cB^m}_\cA$ (in our notation) 
converges to 0 as $m$ goes to $\infty$, while
theorem 3.4 of \cite{R7} shows that the sequence
$\d^\cA_m$ (where it was denoted by $\g_m$)
converges to 0 as $m$ goes to $\infty$.
Thus when we combine this with Proposition \ref{progamn}
we obtain:
\begin{thm}
\label{thmmn}
The height of the bridge $\Pi_{m,n}$ goes to 0 as $m$ 
and $n$ go to $\infty$ simultaneously. Consequently
the length of the bridge $\Pi_{m,n}$ goes to 0 as $m$ 
and $n$ go to $\infty$ simultaneously, and thus
$\mathrm{Prpq}((\cB^m, L^m), (\cB^n, L^n))$
goes to 0 as $m$ 
and $n$ go to $\infty$ simultaneously.
\end{thm} 

We now consider the matricial case. For any natural
number $q$ we apply the constructions
of Section \ref{matricial} to obtain the bridge with 
conditional expectations
\[
\Pi^q_{m,n} = (M_q(\cD), \o_q, E^m_q, E^n_q)
\]
from $M_q(\cB^m)$ to $M_q(\cB^n)$.
From this we then obtain the corresponding maps
$\Phi^m_q$ and $\Phi_q^n$.

We have the actions $\a^q$ of $G$ on $M_q(\cB^m)$ and 
$M_q(\cB^n)$, much as discussed after Theorem \ref{thmht}.
From these actions and the length function $\ell$ we obtain
the slip-norms $L^m_q$ and $L^n_q$. As $q$ varies, these
result in
matrix slip-norms. One shows that $\Pi^q_{m,n}$ is
admissible for $L^m_q$ and $L^n_q$ by arguing in much 
the same way as done after formula \eqref{dphi}. 

We are thus in a position to apply Theorem \ref{thmmatad},
as well as the convergence to 0 indicated above
for the various constants,
to obtain the second main
theorem of this paper:

\begin{thm}
\label{thmmat2q}
With notation as above, we have
\[
\mathrm{length}(\Pi^q_{m,n}) \leq 2q \max\{\breve \g^\cA_m +  
\g^{\cB^n}_\cA, \breve \g^\cA_n +  \g^{\cB^m}_\cA,  
\tilde \d^{\cB^n}_\cA + \tilde\d^\cA_m, \
\tilde \d^{\cB^m}_\cA + \tilde\d^\cA_n \} 
\]
where $\breve\g^\cA_m$ is defined as in formula \ref{gama} while
$\g^{\cB^m}_\cA$ is the $\g^{\cB^m}$ of equation \eqref{gamb}, 
and where $\tilde \d^{\cB^m}_\cA = \tilde \d^{\cB^m}$ 
is defined in Notation \ref{tdel}
while $\tilde\d^\cA_m$ is defined by equation \eqref{dela}, 
and similarly for $n$.
Consequently $\mathrm{length}(\Pi^q_{m,n})$ converges
to 0 as $m$ and $n$ go to $\infty$ simultaneously, 
for each fixed $q$. 
\end{thm}

%%%%%%%%%%%%%%%%%%%%%%%%

\section{Treks}
\label{trek}

\Lat \ defines his propinquity in terms of ``treks''. We will not give here the
precise definition (for which see definition 3.20 of \cite{Ltr2}), but the notion
is quite intuitive. A trek is a finite ``path'' of bridges, so that the ``range''
of the first bridge should be the ``domain'' of the second, etc. The length
of a trek is the sum of the lengths of the bridges in it. The propinquity between
two quantum compact metric spaces is the infimum of the lengths of
all the treks between them. \Lat \ shows in \cite{Ltr2} that propinquity
is a metric on the collection of isometric isomorphism classes of quantum
compact metric spaces. Notably, he proves the striking fact that 
if the propinquity between two quantum compact metric spaces is
0 then they are isometrically isomorphic.

There is an evident trek associated with our second class of examples.
In this section we will briefly examine this trek. Let the notation be as
in the early parts of the previous section. Thus we have $\cA = C(G/H)$,
and the operator algebras $\cB^m$ and $\cB^n$. In Section \ref{app1}
we have the bridge $\Pi_m = (\cA \otimes \cB^m, \o_m)$ from $\cA$ to
$\cB^m$, and the corresponding bridge $\Pi_n$ from $\cA$ to
$\cB^n$. But by reversing the roles of $\cA$ and $\cB^m$ we obtain
a bridge from $\cB^m$ to $\cA$. We do this by still 
viewing $\cA$ and $\cB^m$
as subalgebras of $\cD_m = C(G/H, \cB^m)$, but we now let $\cA$ 
act on the right of $\cD_m$ and we let $\cB^m$ act on the left. We will
denote this bridge by $\cD_m^{-1}$, which is consistent with the 
notation of \Lat \ at the beginning of the proof of proposition 4.7
of \cite{Ltr2}. Of course $\cD_m^{-1}$ has the ``same'' conditional
expectations $E^\cA$ and $E^{\cB^m}$ as those of $\Pi_m$.
We will write $E^\cA$ as $E^\cA_m$ to distinguish it from the
$E^\cA$ from $\cD_n$, which we will denote by $E^\cA_n$.
Then $\cD_m^{-1}$ is a bridge with conditional expectations,
which is easily seen to be admissible for $L^{\cB^m}$ and 
$L^\cA$. It is then easily seen that
\[
\mathrm{length}(\cD_m^{-1}) \ = \ \mathrm{length}(\cD_m)   .
\]
The pair $\G_{m,n} = (\cD_m^{-1}, \ \cD_n)$ then forms a trek
from $\cB^m$ to $\cB^n$, and
\[
\mathrm{length}(\G_{m,n}) \ = \ \mathrm{length}(\cD_m^{-1}) \ + \ 
\mathrm{length}(\cD_n)  .
\]
From Theorem \ref{thmprod} it follows that 
\begin{align*}
\mathrm{length}(\G_{m,n})
&\leq \ 
 \max\{\breve\g^\cA_m, \g^{\cB^m}, \min\{\d^{\cB^m}, \tilde \d^{\cB^m}\} \}   \\
&+ \  \max\{\breve\g^\cA_n, \g^{\cB^n}, \min\{\d^{\cB^n}, \tilde \d^{\cB^n}\} \}  .
\end{align*}
Note that $\tilde\d_m^\cA$ and $\tilde\d_n^\cA$ do not appear in the
above expression, in contrast to their appearance in the estimate
in Theorem \ref{thmht2} for $\mathrm{length}(\Pi_{m,n})$. This
opens the possibility that in some cases
$\mathrm{length}(\G_{m,n})$ gives a
smaller bound for Prpq($\cB^m, \cB^n$) than does
$\mathrm{length}(\Pi_{m,n})$, and, even more, that this might
give examples for which the lengths of certain multi-bridge
treks are strictly smaller that the lengths of any single-bridge
treks. But I have not tried to determine if this happens for
the examples in this paper.

We can view the situation slightly differently as follows. Although
\Lat \ does not mention it, it is natural to define the reach of a
trek as the sum of the reaches of the bridges it contains,
and similarly for the height of a trek. One could then
give a new definition of the length of a trek as simply the
max of its reach and height. This definition is no bigger
that the original definition, and might be smaller. I have not
examined how this might affect the arguments in \cite{Ltr2},
but I imagine that the effect would not be very significant. Anyway,
for the above examples we see from Proposition \ref{prore1}  that
we would have
\[
\mathrm{reach}(\G_{m,n}) \ \leq \ 
\max\{\breve \g^\cA_m, \g_\cA^{\cB^m}\} + 
\max\{\breve \g^\cA_n, \g_\cA^{\cB^n}\},
\]
so that the bound for $\mathrm{reach}(\Pi_{m,n})$ given in
Proposition \ref{prorch2} is no bigger than that above for
$\mathrm{reach}(\G_{m,n})$. But from Theorem \ref{thmprod}
we see that
\[
\mathrm{height}(\G_{m,n}) \ \leq \  \tilde \d^{\cB^m}_\cA
+ \tilde \d^{\cB^n}_\cA
\]
(where $\tilde \d^{\cB^m}_\cA = \tilde \d^{\cB^m}$), 
and this can clearly be less than the right-most bound
for $\mathrm{height}(\Pi_{m,n})$
given in Theorem \ref{thmht2}.

%}   %end large

%%%%%%%%%%%%%%%%%%

%\bibliographystyle{amsplain.bst} 
%\bibliography{ref2014}

\def\dbar{\leavevmode\hbox to 0pt{\hskip.2ex \accent"16\hss}d}
\providecommand{\bysame}{\leavevmode\hbox to3em{\hrulefill}\thinspace}
\providecommand{\MR}{\relax\ifhmode\unskip\space\fi MR }
\providecommand{\MRhref}[2]{%
  \href{http://www.ams.org/mathscinet-getitem?mr=#1}{#2}
}
\providecommand{\href}[2]{#2}

\end{document}